\documentclass[reqno]{amsart}

\usepackage[latin1]{inputenc}
\usepackage{amssymb}
\usepackage{graphicx}
\usepackage{amscd}
\usepackage[hidelinks]{hyperref}
\usepackage{color}
\usepackage{float}
\usepackage{graphics,amsmath,amssymb}
\usepackage{amsthm}
\usepackage{amsfonts}
\usepackage{latexsym}
\usepackage{epsf}
\usepackage{enumerate}
\usepackage{xifthen}
\usepackage{mathrsfs}
\usepackage{dsfont}
\usepackage{makecell}
\usepackage[FIGTOPCAP]{subfigure}
\usepackage{amsmath}
\allowdisplaybreaks[4]
\usepackage{listings}
\usepackage{etoolbox}
\usepackage{fancyhdr}

 \newtheoremstyle{mytheorem}% name % cf. thmtest.tex of AMSLaTeX
 {3pt}%      Space above
 {3pt}%      Space below
 {\slshape}% Body font
 {}%         Indent amount (empty = no indent,
 {\bfseries}% Thm head font
 {.}%        Punctuation after thm head
 { }%        Space after thm head: " " = normal interword space;
 {}%         Thm head spec (can be left empty, meaning `normal')

\numberwithin{equation}{section}

\theoremstyle{theorem}
\newtheorem{theorem}{Theorem}[section]

\theoremstyle{definition}

\newtheorem{example}{Example}[section]

\newcommand{\Keywords}[1]{\ifthenelse{\isempty{#1}}{}{\smallskip \smallskip \noindent \textbf{Keywords}. #1}}
\newcommand{\MSC}[2][2010]{\ifthenelse{\isempty{#2}}{}{\smallskip \smallskip \noindent \textbf{#1MSC}. #2}}
\newcommand{\abstractnote}[1]{\ifthenelse{\isempty{#1}}{}{\smallskip \smallskip \noindent \textsuperscript{\dag}#1}}

\makeatletter
\def\specialsection{\@startsection{section}{1}%
  \z@{\linespacing\@plus\linespacing}{.5\linespacing}%
  {\normalfont}}% NEW
\def\section{\@startsection{section}{1}%
  \z@{.7\linespacing\@plus\linespacing}{.5\linespacing}%
  {\normalfont\scshape}}% NEW
\patchcmd{\@settitle}{\uppercasenonmath\@title}{\Large\boldmath}{}{}
\patchcmd{\@settitle}{\begin{center}}{\begin{flushleft}}{}{}
\patchcmd{\@settitle}{\end{center}}{\end{flushleft}}{}{}
\patchcmd{\@setauthors}{\MakeUppercase}{\normalsize}{}{}
\patchcmd{\@setauthors}{\centering}{\raggedright}{}{}
\patchcmd{\section}{\scshape}{\large\bfseries\boldmath}{}{}
\patchcmd{\subsection}{\bfseries}{\bfseries\boldmath}{}{}
\renewcommand{\@secnumfont}{\bfseries}
\patchcmd{\@startsection}{\@afterindenttrue}{\@afterindentfalse}{}{}
\patchcmd{\abstract}{\leftmargin3pc}{\leftmargin1pc}{}{}

\def\maketitle{\par
  \@topnum\z@ % this prevents figures from falling at the top of page 1
  \@setcopyright
  \thispagestyle{empty}% this sets first page specifications
  \ifx\@empty\shortauthors \let\shortauthors\shorttitle
  \else \andify\shortauthors
  \fi
  \@maketitle@hook
  \begingroup
  \@maketitle
  \toks@\@xp{\shortauthors}\@temptokena\@xp{\shorttitle}%
  \toks4{\def\\{ \ignorespaces}}% defend against questionable usage
  \edef\@tempa{%
    \@nx\markboth{\the\toks4
      \@nx\MakeUppercase{\the\toks@}}{\the\@temptokena}}%
  \@tempa
  \endgroup
  \c@footnote\z@
  \@cleartopmattertags
}
\makeatother

\title{New congruences for $\ell$-regular overpartitions}

\author[S. Chern]{Shane Chern}
\address{Department of Mathematics, Pennsylvania State University, University Park, PA 16802, USA}
\email{shanechern@psu.edu}

\date{}

\begin{document}

{\footnotesize\noindent \textit{Integers} \textbf{17} (2017), Paper No. A22, 8 pp.}

\bigskip \bigskip

\maketitle

\begin{abstract}

Recently, Shen (2016) and Alanazi et al. (2016) studied the arithmetic properties of the $\ell$-regular overpartition function $\overline{A}_\ell (n)$, which counts the number of overpartitions of $n$ into parts not divisible by $\ell$. In this note, we will present some new congruences modulo $5$ when $\ell$ is a power of $5$.

\Keywords{Congruence, overpartition, $\ell$-regular partition.}

\MSC{Primary 05A17; Secondary 11P83.}
\end{abstract}

\section{Introduction}

A \textit{partition} of a natural number $n$ is a nonincreasing sequence of positive integers whose sum is $n$. For example, $6=3+2+1$ is a partition of $6$. Let $p(n)$ denote the number of partitions of $n$. We also agree that $p(0)=1$. It is well-known that the generating function of $p(n)$ is given by
$$\sum_{n\ge 0}p(n)q^n=\frac{1}{(q;q)_\infty},$$
where we adopt the standard notation
$$(a;q)_\infty=\prod_{n\ge 0}(1-aq^n).$$

For any positive integer $\ell$, a partition is called \textit{$\ell$-regular} if none of its parts are divisible by $\ell$. Let $b_\ell(n)$ denote the number of $\ell$-regular partitions of $n$. We know that its generating function is
$$\sum_{n\ge 0}b_\ell(n)q^n=\frac{(q^\ell;q^\ell)_\infty}{(q;q)_\infty}.$$
On the other hand, an \textit{overpartition} of $n$ is a partition of $n$ in which the first occurrence of each distinct part can be overlined. Let $\overline{p}(n)$ be the number of overpartitions of $n$. We also know that the generating function of $\overline{p}(n)$ is
$$\sum_{n\ge 0}\overline{p}(n)q^n=\frac{(-q;q)_\infty}{(q;q)_\infty}=\frac{(q^2;q^2)_\infty}{(q;q)_\infty^2}.$$
Many authors have studied the arithmetic properties of $b_\ell(n)$ and $\overline{p}(n)$. We refer the interested readers to the ``Introduction'' part of \cite{She2016} and references therein for detailed description.

In \cite{Lov2003}, Lovejoy introduced a function $\overline{A}_\ell(n)$, which counts the number of overpartitions of $n$ into parts not divisible by $\ell$. According to Shen \cite{She2016}, this type of partition can be named as \textit{$\ell$-regular overpartition}. He also obtained the generating function of $\overline{A}_\ell(n)$, that is,
\begin{equation}\label{eq:gf}
\sum_{n\ge 0}\overline{A}_\ell(n)q^n=\frac{(q^\ell;q^\ell)_\infty^2(q^2;q^2)_\infty}{(q;q)_\infty^2(q^{2\ell};q^{2\ell})_\infty}.
\end{equation}
Meanwhile, he presented several congruences for $\overline{A}_3(n)$ and $\overline{A}_4(n)$. For $\overline{A}_3(n)$, he got
\begin{align*}
\overline{A}_3(4n+1)&\equiv 0 \pmod{2},\\
\overline{A}_3(4n+3)&\equiv 0 \pmod{6},\\
\overline{A}_3(9n+3)&\equiv 0 \pmod{6},\\
\overline{A}_3(9n+6)&\equiv 0 \pmod{24}.
\end{align*}
More recently, Alanazi et al. \cite{AMS2016} further studied the arithmetic properties of $\overline{A}_{\ell} (n)$ under modulus $3$ when $\ell$ is a power of $3$. They also gave some congruences satisfied by $\overline{A}_{\ell} (n)$ modulo $2$ and $4$.

In this note, our main purpose is to study the arithmetic properties of $\overline{A}_\ell(n)$ when $\ell$ is a power of $5$. When $\ell=5$ and $25$, we connect $\overline{A}_\ell (n)$ with $r_4(n)$ and $r_8(n)$ respectively, where $r_k(n)$ denotes the number of representations of $n$ by $k$ squares. The method for $\ell=5$ also applies to other prime $\ell$. When $\ell=125$, we show that
$$\overline{A}_{125}(25n)\equiv \overline{A}_{125}(625n)\pmod{5}.$$
This can be regarded as an analogous result of the following congruence for overpartition function $\overline{p}(n)$ proved by Chen et al. (see \cite[Theorem 1.5]{Che2015})
$$\overline{p}(25n)\equiv \overline{p}(625n)\pmod{5}.$$
When $\ell=5^\alpha$ with $\alpha\ge 4$, we obtain new congruences similar to a result of Alanazi et al. (see \cite[Theorem 3]{AMS2016}), which states that $\overline{A}_{3^\alpha}(27n+18)\equiv 0\pmod{3}$ holds for all $n\ge 0$ and $\alpha\ge 3$.

\section{New congruence results}

\subsection{$\ell=5$}

One readily sees from the binomial theorem that for any prime $p$,
\begin{equation}\label{eq:qp}
(q;q)_\infty^p\equiv (q^p;q^p)_\infty \pmod{p}.
\end{equation}
Setting $p=5$ and applying it to \eqref{eq:gf}, we have
\begin{equation}\label{eq:gf5}
\sum_{n\ge 0}\overline{A}_5(n)q^n\equiv \left(\frac{(q;q)_\infty^2}{(q^2;q^2)_\infty}\right)^{4} \pmod{5}.
\end{equation}

Now let
$$\varphi(q):=\sum_{n=-\infty}^{\infty}q^{n^2}.$$
It is well-known that (see \cite[p. 37, Eq. (22.4)]{Ber1991})
$$\varphi(-q)=\frac{(q;q)_\infty^2}{(q^2;q^2)_\infty}.$$
We therefore have

\begin{theorem}\label{th:l5}
For any positive integer $n$, we have
\begin{equation}
\overline{A}_5(n)\equiv\begin{cases}
r_{4}(n) & \text{if $n$ is even}\\
-r_{4}(n) & \text{if $n$ is odd}
\end{cases} \pmod{5}.
\end{equation}
\end{theorem}

We know from \cite[Theorem 3.3.1]{Ber2006} that $r_4(n)=8d^*(n)$ where
$$d^*(n)=\sum_{d\mid n,\, 4\nmid d}d.$$
Let $p\ne 5$ be an odd prime and $k$ be a nonnegative integer. One readily verifies that
$$\sum_{i=0}^{4k+3}p^i\equiv (k+1)\sum_{i=1}^4 i\equiv 0 \pmod{5}.$$
Note also that $d^*(n)$ is multiplicative. Thus we conclude
\begin{theorem}
Let $p\ne 5$ be an odd prime and $k$ be a nonnegative integer. Let $n$ be a nonnegative integer. We have
\begin{equation}
\overline{A}_5(p^{4k+3}(pn+i))\equiv 0 \pmod{5},
\end{equation}
where $i\in\{1,2,\ldots,p-1\}$.
\end{theorem}

\begin{example}
If we take $p=3$, $k=0$, and $i=1$, then
\begin{equation}
\overline{A}_5(81n+27)\equiv 0 \pmod{5}
\end{equation}
holds for all $n\ge 0$.
\end{example}

We also note that if an odd prime $p$ is congruent to $9$ modulo $10$, then $1+p\equiv 0 \pmod{5}$. We therefore have $\overline{A}_5(p(pn+i))\equiv 0 \pmod{5}$ for $i\in\{1,2,\ldots,p-1\}$. On the other hand, if we require $1+p+p^2 \equiv 0 \pmod{5}$, then $(2p+1)^2\equiv -3 \pmod{5}$. However, since $(-3|5)=-1$ (here $(*|*)$ denotes the Legendre symbol), such $p$ does not exist. The above observation yields

\begin{theorem}
Let $p\equiv 9 \pmod{10}$ be a prime and $n$ be a nonnegative integer. We have
\begin{equation}
\overline{A}_5(p(pn+i))\equiv 0 \pmod{5},
\end{equation}
where $i\in\{1,2,\ldots,p-1\}$.
\end{theorem}

\begin{example}
If we take $p=19$ and $i=1$, then
\begin{equation}
\overline{A}_5(361n+19)\equiv 0 \pmod{5}
\end{equation}
holds for all $n\ge 0$.
\end{example}

We should mention that this method also applies to other primes $\ell$. In fact, if we set $p=\ell$ in \eqref{eq:qp} and apply it to \eqref{eq:gf}, then
\begin{equation}
\overline{A}_\ell(n)\equiv\begin{cases}
r_{\ell-1}(n) & \text{if $n$ is even}\\
-r_{\ell-1}(n) & \text{if $n$ is odd}
\end{cases} \pmod{\ell}.
\end{equation}
Recall that the explicit formulas of $r_2(n)$ and $r_6(n)$ are also known. From \cite[Theorems 3.2.1 and 3.4.1]{Ber2006}, we have
$$r_2(n)=4\sum_{d\mid n}\chi(d),$$
and
$$r_6(n)=16\sum_{d\mid n}\chi(n/d)d^2-4\sum_{d\mid n}\chi(d)d^2,$$
where
$$\chi(n)=\begin{cases}
1 & n\equiv 1 \pmod{4},\\
-1 & n\equiv 3 \pmod{4},\\
0 & \text{otherwise.}
\end{cases}$$
Through a similar argument, we conclude that

\begin{theorem}
For any nonnegative integers $n$, $k$, odd prime $p$, and $i\in\{1,2,\ldots,p-1\}$, we have
\begin{align*}
\overline{A}_3(p^{2k+1}(pn+i))&\equiv 0 \pmod{3}\quad\text{where $p\equiv 3 \pmod{4}$},\\
\overline{A}_3(p^{3k+2}(pn+i))&\equiv 0 \pmod{3}\quad\text{where $p\equiv 1 \pmod{4}$},\\
\overline{A}_7(p^{6k+5}(pn+i))&\equiv 0 \pmod{7}\quad\text{where $p\ne 7$}.
\end{align*}
\end{theorem}

\subsection{$\ell=25$}

Analogous to the congruences under modulus $3$ for $\overline{A}_{9} (n)$ obtained by Alanazi et al. \cite{AMS2016}, we will present some arithmetic properties of $\overline{A}_{25} (n)$ modulo $5$. Rather than using the technique of dissection identities, we build a connection between $\overline{A}_{25} (5n)$ and $r_8(n)$ and then apply the explicit formula of $r_8(n)$. It is necessary to mention that this method also applies to the results of Alanazi et al. as the following relation holds
$$\overline{A}_9(n)\equiv (-1)^n r_8(n) \pmod{3}.$$

Note that
\begin{align*}\sum_{n\ge 0}\overline{A}_{25} (n)q^n&=\frac{(q^{25};q^{25})_\infty^2}{(q^{50};q^{50})_\infty}\frac{(q^2;q^2)_\infty}{(q;q)_\infty^2}\\
&\equiv \left(\frac{(q^{5};q^{5})_\infty^2}{(q^{10};q^{10})_\infty}\right)^5\left(\sum_{n\ge 0}\overline{p}(n)q^n\right) \pmod{5}.
\end{align*}
Extracting powers of the form $q^{5n}$ from both sides and replacing $q^5$ by $q$, we have
\begin{align*}
\sum_{n\ge 0}\overline{A}_{25} (5n)q^n&\equiv \left(\frac{(q;q)_\infty^2}{(q^{2};q^{2})_\infty}\right)^5\left(\sum_{n\ge 0}\overline{p}(5n)q^n\right)\\
&\equiv \varphi(-q)^8 \pmod{5}.
\end{align*}
Here we use the following celebrated result due to Treneer \cite{Tre2006}
$$\sum_{n\ge 0}\overline{p}(5n)q^n\equiv \varphi(-q)^3 \pmod{5}.$$

It is known that the explicit formula of $r_8(n)$ (see \cite[Theorem 3.5.4]{Ber2006}) is given by
$$r_8(n)=16(-1)^n \sigma_3^-(n),$$
where
$$\sigma_3^-(n)=\sum_{d\mid n}(-1)^d d^3.$$
We therefore conclude that

\begin{theorem}\label{th:25}
For any positive integer $n$, we have
\begin{equation}
\overline{A}_{25}(5n)\equiv \sigma_3^-(n) \pmod{5}.
\end{equation}
\end{theorem}

One also readily deduces several congruences from Theorem \ref{th:25}.

\begin{theorem}
Let $p\ne 5$ be an odd prime and $k$ be a nonnegative integer. Let $n$ be a nonnegative integer. We have
\begin{equation}
\overline{A}_{25}(5p^{4k+3}(pn+i))\equiv 0 \pmod{5},
\end{equation}
where $i\in\{1,2,\ldots,p-1\}$.
\end{theorem}

\begin{proof}
It is easy to see that
$$\sigma_3^-(p^{4k+3})=-\sum_{i=0}^{4k+3}p^{3i}\equiv -(k+1)\sum_{i=1}^4 i\equiv 0 \pmod{5}.$$
Note also that $\sigma_3^-(n)$ is multiplicative. The theorem therefore follows.
\end{proof}

Furthermore, if $p\equiv 9 \pmod{10}$ is a prime, then $1+p^3\equiv 0 \pmod{5}$. This yields

\begin{theorem}
Let $p\equiv 9 \pmod{10}$ be a prime and $n$ be a nonnegative integer. We have
\begin{equation}
\overline{A}_{25}(5p(pn+i))\equiv 0 \pmod{5},
\end{equation}
where $i\in\{1,2,\ldots,p-1\}$.
\end{theorem}

\subsection{$\ell=125$}

From the generating function \eqref{eq:gf}, we have
$$\sum_{n\ge 0}\overline{A}_{125} (n)q^n=\frac{(q^{125};q^{125})_\infty^2(q^2;q^2)_\infty}{(q;q)_\infty^2(q^{250};q^{250})_\infty}=\varphi(-q^{125})\sum_{n\ge 0}\overline{p}(n)q^n.$$
Extracting terms of the form $q^{125n}$ and replacing $q^{125}$ by $q$, we have
$$\sum_{n\ge 0}\overline{A}_{125} (125n)q^n=\varphi(-q)\sum_{n\ge 0}\overline{p}(125n)q^n.$$

According to \cite[Eq. (5.3)]{Che2015}, we know that
$\overline{p}(125(5n\pm 1))\equiv 0\pmod{5}$.
We also know from \cite[p. 49, Corollary (i)]{Ber1991} that
$$\varphi(-q)=\varphi(-q^{25})-2qM_1(-q^{5})+2q^4M_2(-q^{5}),$$
where $M_1(q)=f(q^3,q^7)$ and $M_2(q)=f(q,q^9)$. Here $f(a,b)$ is the Ramanujan's theta function defined as
$$f(a,b):=\sum_{n=-\infty}^{\infty} a^{n(n+1)/2}b^{n(n-1)/2}=(-a;ab)_\infty(-b;ab)_\infty(ab;ab)_\infty.$$
Now we extract terms involving $q^{5n}$ from $\sum_{n\ge 0}\overline{A}_{125} (125n)q^n$ and replace $q^5$ by $q$, then
$$\sum_{n\ge 0}\overline{A}_{125} (625n)q^n\equiv \varphi(-q^5)\sum_{n\ge 0}\overline{p}(625n)q^n\pmod{5}.$$
Finally, we use the following congruence from \cite[Theorem 1.5]{Che2015} for $\overline{p}(n)$
$$\overline{p}(25n)\equiv \overline{p}(625n)\pmod{5},$$
and obtain
$$\sum_{n\ge 0}\overline{A}_{125} (625n)q^n\equiv \varphi(-q^5)\sum_{n\ge 0}\overline{p}(25n)q^n\pmod{5},$$
which coincides with
$$\sum_{n\ge 0}\overline{A}_{125} (25n)q^n\equiv \varphi(-q^5)\sum_{n\ge 0}\overline{p}(25n)q^n\pmod{5}.$$
We therefore conclude

\begin{theorem}
For any nonnegative integer $n$, we have
\begin{equation}
\overline{A}_{125} (25n)\equiv \overline{A}_{125} (625n) \pmod{5}.
\end{equation}
\end{theorem}

\subsection{$\ell=5^\alpha$ with $\alpha\ge 4$}

We know from \eqref{eq:gf} that in this case
$$\sum_{n\ge 0}\overline{A}_{5^\alpha} (n)q^n=\varphi\left(-q^{5^\alpha}\right)\sum_{n\ge 0}\overline{p}(n)q^n.$$
Note that for $\alpha\ge 4$, $\varphi\left(-q^{5^\alpha}\right)$ is a function of $q^{625}$. Hence $\overline{A}_{5^\alpha} (625n+ 125)$ (resp. $\overline{A}_{5^\alpha} (625n+ 500)$) is a linear combination of values of $\overline{p}(625n+ 125)$ (resp. $\overline{p}(625n+ 500)$). Thanks to \cite[Eq. (5.3)]{Che2015}, we know that $\overline{p}(625n+ 125)\equiv 0\pmod{5}$ and $\overline{p}(625n+ 500)\equiv 0\pmod{5}$ hold for all $n\ge 0$. Hence

\begin{theorem}\label{th:alpha4.1}
For any nonnegative integer $n$ and positive integer $\alpha\ge 4$, we have
\begin{equation}
\overline{A}_{5^\alpha}(625n+ i)\equiv 0\pmod{5},
\end{equation}
where $i=125$ and $500$.
\end{theorem}

Note also that for $\alpha\ge 2$, extracting terms of the form $q^{25n}$ and replacing $q^{25}$ by $q$, we obtain
$$\sum_{n\ge 0}\overline{A}_{5^\alpha} (25n)q^n=\varphi\left(-q^{5^{\alpha-2}}\right)\sum_{n\ge 0}\overline{p}(25n)q^n.$$
On the other hand, we extract terms involving $q^{625n}$ from $\sum_{n\ge 0}\overline{A}_{5^{\alpha+2}} (n)q^n$ and replace $q^{625}$ by $q$, then
$$\sum_{n\ge 0}\overline{A}_{5^{\alpha+2}} (625n)q^n=\varphi\left(-q^{5^{\alpha-2}}\right)\sum_{n\ge 0}\overline{p}(625n)q^n.$$
Thanks again to \cite[Theorem 1.5]{Che2015}, which states that $\overline{p}(25n)\equiv \overline{p}(625n)\pmod{5}$ for all $n\ge 0$, we conclude

\begin{theorem}\label{th:alpha4.2}
For any nonnegative integer $n$, we have
\begin{equation}
\overline{A}_{5^\alpha}(25n)\equiv \overline{A}_{5^{\alpha+2}} (625n)\pmod{5}
\end{equation}
for all $\alpha\ge 2$.
\end{theorem}

It follows from Theorems \ref{th:alpha4.1} and \ref{th:alpha4.2} that

\begin{theorem}\label{th:alpha4.3}
For any nonnegative integer $n$ and positive integer $\alpha\ge 4$, we have
\begin{equation}
\overline{A}_{5^\alpha}(5^{2j}(625n+i))\equiv 0\pmod{5}
\end{equation}
for all integers $0\le j\le (\alpha-4)/2$, where $i=125$ and $500$.
\end{theorem}

\subsection*{Acknowledgements}

The author would like to express gratitude to James A. Sellers
for some interesting discussions as well as thank the reviewer for several helpful
comments.

\bibliographystyle{amsplain}

\end{document}